\documentclass[11pt]{amsart}
\usepackage{amssymb,hyperref}
\usepackage{enumerate}

\hypersetup{pdfpagemode=FullScreen,colorlinks=true}

\newtheorem{thm}{Theorem}[section]

\newtheorem{definition}[thm]{Definition}

\newcommand{\bc}{{\mathbb C}}

\newcommand{\bh}{{\mathbb H}}
\newcommand{\bn}{{\mathbb N}}

\newcommand{\ra}{\rightarrow}

\def\cal{\mathcal}

\newcommand{\pslc}{\operatorname{PSL}_2 \bc}
\newcommand{\ie}{i.e.\ }
\newcommand{\Fr}{\operatorname{Fr}}
\def\hyperbolic{\bh}
\begin{document}
\title[Primitive stable]
{Primitive stable closed hyperbolic 3-manifolds}
\author{Inkang Kim, Cyril Lecuire and Ken'ichi Ohshika}
\date{}

\begin{abstract}
We show that closed 3-manifolds with high Heegaard distance and bounded subsurface Heegaard distance are
primitive stable when they are regarded as representations from the free group corresponding to the handlebody.
This implies that any point on the boundary of Schottky space can be approximated by primitive stable representations corresponding to closed hyperbolic 3-manifolds.
\end{abstract}

\footnotetext[1]{I. Kim gratefully acknowledges the partial support
of NRF grant  (2010-0024171).}
\footnotetext[2]{C. Lecuire gratefully acknowledges the support of  Indo-French Research Grant 4301-1.}
\footnotetext[3]{K. Ohshika gratefully acknowledges the support of the JSPS grant-in-aid for scientific research (A) 22244005.}
\maketitle

\section{Introduction}\label{intro}
It is well known that any closed orientable 3-manifold admits a Heegaard decomposition along a splitting
surface $\Sigma$, that is, a decomposition with $M=H^1\cup H^2$, $H^1$ and $H^2$ being homeomorphic handlebodies and $H^1\cap H^2=\partial
H^1=\partial H^2=\Sigma$.
For a Heegaard splitting as above and $j=1,2$, we consider a subset $\cal D_j(\Sigma)\subset \cal
C^1(\Sigma)$ in the curve graph of $\Sigma$ consisting of all essential closed
curves that bound  discs in $H^j$, which we call meridians.
Masur and Minsky \cite{MaM2} proved that these subsets are  quasi-convex in the curve graph.
The Hempel distance or the Heegaard distance of the splitting $M=H^1 \cup H^2$ is defined to
be
$$d(H^1,H^2)=\min\{d(c_1,c_2)|c_i\in \cal D_i(\Sigma)\}.$$

The Hempel distance reflects some properties of 3-manifolds.
Haken \cite{Haken} showed that if $d(H^1, H^2) \geq 1$, then $M$ is irreducible.
Hempel \cite{Hempel} proved that if $M$ contains
an incompressible torus or is a Seifert fibred manifold, then $d(H^1,H^2)\leq 2$ for any splitting
$(H^1,H^2;\Sigma)$.
Combined with the Geometrisation Theorem due to Thurston and Perelman, this implies that if $d(H^1, H^2) \geq 3$, then $M$ is hyperbolic.


Recently Minsky \cite{Mi2} introduced a notion of primitive
stability for $\pslc$ representations of free groups. The main interest of primitive stable representations is that they form an open subset of the character variety which is bigger than the set of convex cocompact representations and on which the group of outer automorphisms acts properly discontinuously.
 Let us briefly recall the definition of primitive stable representations.

In a free group, an element is called \emph{primitive} if it can be a member of a free generating set. A representation $\rho:F\rightarrow \pslc$ of a free group is {\em primitive stable} if it has the following property. Any $\rho$-equivariant map from a Cayley graph of $F$ to $\hyperbolic^3$ maps the geodesics defined by primitive elements to uniform quasi-geodesics.

In \cite{MiM}, Minsky and Moriah constructed lattices which are the images of primitive stable representations and asked whether any lattice is the image of such a primitive stable representation. In this note, we shall give further examples of lattices which are images of primitive stable representations  by giving some sufficient conditions for primitive stability for closed hyperbolic 3-manifolds.
To be more concrete, we shall show that under some conditions, every Heegaard splitting with Hempel distance large enough and with some boundedness condition (boundedness with regard to subsurfaces) give two primitive stable representations of the free groups corresponding  to the two handlebodies constituting the splitting.
As an application of our main result, we shall also show that any boundary point of a Schottky space is a limit of primitive stable representations corresponding to closed hyperbolic 3-manifolds.

\section{Statement of Main Theorem}
Throughout this paper we assume all manifolds to be closed and orientable, and all Heegaard surfaces to have genus at least $2$.

Let $\Sigma$ be a closed surface.
An essential simple closed curve in $\Sigma$ is simply called a curve on $\Sigma$.
A subsurface $Y$ of $\Sigma$ is said to be essential when every frontier component of $Y$ is essential in $\Sigma$.
For a curve $c$ on $\Sigma$ and an essential subsurface $Y$ of $\Sigma$, we define the projection of $c$ to $Y$, which we denote by $\pi_Y(c)$ as follows.
If $c$ does not intersect $Y$ essentially, then we define $\pi_Y(c)$ to be the empty set.
If $c$ intersects $Y$ essentially, then we define $\pi_Y(c)$ to be the set of simple closed curves on $Y$ obtained from the components of $c \cap Y$ by connecting their endpoints by arcs on
the frontier of $ Y$.

\begin{definition}
\label{bounded}
Let $M=H^1 \cup_\Sigma H^2$ be a Heegaard splitting.
Let $\mathcal D_1$ and $\mathcal D_2$ be the set of isotopy classes of meridians in $H^1$ and $H^2$ respectively, regarded as subsets in the curve graph of $\Sigma$.
We call these the meridian complexes for $H^1$ and $H^2$.
Let $Y$ be an essential subsurface of $\Sigma$.
Then the $Y$-Heegaard distance of $H^1 \cup_\Sigma H^2$ is the distance in the curve graph of $Y$ between $\pi_Y(\mathcal D_1)$ and $\pi_Y(\mathcal D_2)$.
\end{definition}

\begin{definition}
We say that a Heegaard splitting $M=H^1 \cup_\Sigma H^2$ has {\em $R$-bounded subsurface distance} when  for any essential subsurface $Y$ of $\Sigma$ the $Y$-Heegaard distance of $H^1 \cup_\Sigma H^2$ is bounded by $R$.
\end{definition}

The main theorem of this note is the following.


\begin{thm}
\label{main}
For any $R$, there exists $K$ depending only on $R$ and the genus $g$ as follows.
For any $3$-manifold admitting a genus-$g$ Heegaard splitting $M=H^1 \cup_\Sigma H^2$ whose Heegaard distance is greater than $K$ and which has $R$-bounded subsurface distance, the manifold $M$ is hyperbolic and the representation $\iota_*: \pi_1(H^j) \rightarrow \pi_1(M) \subset \pslc$ is primitive stable for $j=1,2$.
\end{thm}

Our proof of this theorem relies on the result of Namazi \cite{Na} on model manifolds for Heegaard splittings with uniformly bounded $Y$-Heegaard distance and on the characterisation of primitive stable discrete and faithful representations given in \cite{JKOL}. As was mentioned before, as long as $K\geq 3$, the manifold $M$ is hyperbolic by the Geometrisation Theorem, but Namazi gave an alternative  proof of the hyperbolicity of $M$ for $K$ large enough as in our statement, which does not use the full Geometrisation Theorem. See \cite{Na}.

\section{Criterion for primitive stability}
Let $F$ be a non-abelian free group.
Fix some symmetric generator system and consider the Cayley graph $C(F)$ with respect to the generator system.
Given a representation  $\rho : F \rightarrow\pslc$ and a base point $o\in
\bh^3$, there is a unique $\rho$-equivariant map
$\tau_{\rho,o}:C(F) \ra \bh^3$ sending the origin $e$ of $C(F)$
to $o$ and taking each edge to a geodesic segment \cite{Fl}.
 A representation $\rho : F \rightarrow \pslc$ is {\it primitive stable} if
there are constants $K,\delta$ such that
$\tau_{\rho,o}$ takes all bi-infinite geodesics in $C(F)$ determined by primitive elements to $(K,\delta)$-quasi-geodesics in $\bh^3$.
This definition is independent of the choice of the base point $o\in \bh^3$, which we can easily verify by changing $K$ and $\delta$.

A measured lamination (or a simple closed curve) $\lambda$  on the boundary of a handlebody is said to be {\it disc-busting} if there exists
$\eta>0$ such that  $i(\partial D,
\lambda)>\eta$ for any compressing disc $D$.
Otherwise $\lambda$ is called {\em disc-dodging}.
In \cite{JKOL}, a complete criterion for primitive stability for faithful discrete representations of the fundamental group of a handlebody $H$ to $\pslc$ was given.
\begin{thm}[Jeon-Kim-Ohshika-Lecuire]
\label{JKOL}
Let $\rho$ be a discrete, faithful and geometrically infinite
representation possibly with parabolics such that the non-cuspidal part $H_0$ of $H=\bh^3/\rho(F)$ is the union of a relative compact core $C_0$ and finitely many ends $E_i$ facing $S_i\subset \partial H$.
Then the representation $\rho$ is primitive stable if and only if every parabolic curve is disc-busting, and every geometrically infinite end $E_i$ has  ending lamination $\lambda_i$ which is disc-busting on $\partial H$.
In particular, if $\rho$ is purely loxodromic, then it is primitive stable.
\end{thm}

\section{Proof of the main theorem}
The proof of the main theorem relies on the work of Namazi \cite{Na} on geometric models associated with Heegaard splittings and on Theorem \ref{JKOL} (more specifically the last sentence of the statement, i.e. \cite[Theorem 1.2]{JKOL}).
We shall prove the contrapositive of our statement.

Fix $R>0$ and consider a sequence $\{M_i\}$ of hyperbolic 3-manifolds with genus $g$ Heegaard splittings $M_i=H^1_i \cup_{\Sigma_i} H^2_i$ having $R$-bounded subsurface distances and Hempel distances $K_i\longrightarrow\infty$. By \cite[Main Theorem 1]{Na}, for sufficiently large $K_0$, if $K_i\geq K_0$, there is a  negatively curved model manifold $N_i$ homeomorphic to $M_i$ whose sectional curvature lies in $(-1-\epsilon, -1+\epsilon)$, where $\epsilon$ depends on $K_0$ and $g$, and goes to $0$ as $K_0 \rightarrow \infty$ fixing $g$, and whose injectivity radii are bounded from below by a uniform positive constant depending only on $K_0$ and $g$. By \cite[Corollary 12.2]{Na}, the metric on $N_i$ is close to the metric on $M_i$ up to the third derivative (and gets closer and closer as $i$ goes to $\infty$).

The model manifold $N_i$ is constructed as follows (see \cite{Na}, p. 173-175). Let $\alpha^1_i$ and $\alpha^2_i$ be pants decompositions of $\partial H^1_i =\partial H^2_i=\Sigma_i$ such that each component of $\alpha^j_i$ bounds a compressing disc in $H^j_i$ and a component of $\alpha^1_i$ and a component of $\alpha^2_i$ realise the Hempel distance of the splitting (i.e. their distance in the curve graph of $\Sigma\approx\Sigma_i$ is equal to the Hempel distance of the splitting $M_i=H^1_i \cup_{\Sigma_i} H^2_i$). Pick a point $\tau^j_i$ in the Teichm\"{u}ller space of $\Sigma$ such that the length of $\alpha^j_i$ with respect to $\tau^j_i$ is less than a fixed constant $B_0$ (for example the Bers constant), and let $m^1_i$ be a convex cocompact hyperbolic structure on $H\approx H^1_i$ whose conformal structure at infinity is $\tau^2_i$, and $m^2_i$  a convex cocompact hyperbolic structure on $H\approx H^2_i$ whose conformal structure at infinity is $\tau^1_i$. Then $N_i$ is obtained
  by pasting a large piece of the convex core of $C^1_i:=(H^1_i, m^1_i)$ with a large piece of the convex core of $C^2_i:=(H^2_i, m^2_i)$ using an $I$-bundle homeomorphic to $\Sigma_i\times I$  (see \cite{Na}, p. 173-175). The metric on this $I$-bundle comes from a deformation of a piece of a doubly degenerate hyperbolic manifold whose description is not relevant to our proof.

It follows from this construction that there are sequences of base points $x^j_i\subset N_i$ and $y^j_i\subset C^j_i$ such that the sequences $\{(N_i,x^j_i)\}$ and $\{(C^j_i,y^j_i)\}$ have the same limit in the pointed Hausdorff-Gromov topology for $j=1,2$. Since the metric on $N_i$ gets closer and closer to the metric on $M_i$, there are sequences of base points $z^j_i\subset M_i$ such that  the sequences $\{(M_i,z^j_i)\}$ and $\{(C^j_i,y^j_i)\}$ have the same limit in the pointed Hausdorff-Gromov topology for $j=1,2$.

Since the Heegaard splittings $M_i=H^1_i \cup H^2_i$ have $R$-bounded subsurface distances, it follows from \cite[Corollary 6.4]{Na} that, passing to a subsequence, $\tau^j_i$ tends in the Thurston compactification to an arational  lamination $\lambda^j$ in the Masur domain on the boundary of $H\approx H^1_i\approx H^2_i$. It follows then from \cite{Oh} or \cite{NS2} that a subsequence of $\{C^j_i\}$ (under the identification of $H$ with $H^j_i$ as above) converges algebraically (and geometrically by the Covering Theorem \cite{Ca2}) to a singly degenerate hyperbolic open handlebody with $|\lambda^j|$ its ending lamination.
Therefore, by pulling back a marking on a compact core of the geometric limit of $\{(M_i, z_i^j)\}$, we have  a marking $\phi^j_i : F_g \rightarrow \pi_1(H^j_i) \subset \pi_1(M_i) \subset \pslc$ which converges algebraically as $i \rightarrow \infty$ after passing to a subsequence.

Let $\phi^j_\infty$ be the limit of $\{\phi^j_i\}$.
As we explained before, $H_\infty^j=\bh^3/\phi_\infty^j(F_g)$ is isometric to the geometric limit of $\{(C^j_i,y^j_i)\}$ which is a singly degenerate hyperbolic open handlebody. By Theorem \ref{JKOL}, $\phi^j_\infty$ is primitive stable for $j=1,2$. Since the primitive stability is an open condition, this implies that $\phi^j_i$ is also primitive stable for sufficiently large $i$. This concludes the proof of our main theorem, for we have proved its contrapositive.

\section{Applications}
In this section, we shall present two applications of our main theorem.
In the first, we shall give  concrete examples of primitive stable representations obtained by Theorem \ref{main}.
In the second, we shall show that every point on the boundary of the Schottky space can be approximated by primitive stable representations corresponding to closed hyperbolic 3-manifolds.

Consider a standard genus $g$ Heegaard splitting  of  $\sharp_g S^1 \times S^2=H^1 \cup H^2$ obtained by doubling.
Let $\eta: \partial H^2 \rightarrow \partial H^1$ denote the identification of the two boundaries in this splitting.
For a mapping class $\phi$ of $S=\partial H^2$, we denote by $H^1 \cup_\phi H^2$  the Heegaard splitting of a 3-manifold obtained by pasting $\partial H^2$ to $\partial H^1$ by $\eta \circ \phi$ instead of $\eta$.

A pseudo-Anosov mapping class on the compressible boundary $S$ of a 3-manifold
$M$ is said to {\em partially extend} if its representative extends to a homeomorphism on
a non-trivial compression body inside $M$, whose exterior boundary is $S$ (see \cite{BJM}).

\begin{thm}\label{pseudo} Set $M_i=H^1 \cup_{\phi^i} H^2$, where $\phi:\partial H^2
\ra \partial H^2$ is a pseudo-Anosov mapping class  no power of which
extends partially  to $H^2$. Then
there is some $N$ such that $M_i$ is primitive stable for every $i \geq N$.
\end{thm}

By \cite{BJM}, no power of  $\phi$ partially extends if and only if the  stable  lamination of $\phi$ lies in
the Masur domain of $\partial H^2$.
It was proved in \cite{BJM} this also implies that the unstable lamination of $\phi$ also lies in the Masur domain of $\partial H_2$.

\begin{proof}
This could be proved just by replacing the arguments from \cite{Na} in the proof of Theorem \ref{main} with arguments from \cite{NS2}.
Instead, we shall prove that  this is a special case of Theorem \ref{main}.
Let $\mathcal D^1_i$ and $\mathcal D^2_i$ be the meridian complexes for the Heegaard splitting $M_i=H^1 \cup_{\phi^i} H^2$.
Let $\gamma_i$ be a tight geodesic in the curve graph of $S$ connecting a point in $\mathcal D^1_i$ and a point in $\mathcal D^2_i$.
Then, as was shown in Namazi-Souto \cite{NS2}, $\gamma_i$ converges, uniformly on any compact set, to a tight geodesic in the curve complex connecting the unstable and stable laminations which are regarded as points at infinity of the curve graph, as $i \rightarrow \infty$.
This implies, by an argument of \cite{MaM}, that there are $R>0$ and $n_0 \in \bn$ such that if $i \geq n_0$, for any essential subsurface $Y$ on $S$,  the projections of the endpoints of $\gamma_i$ to the curve complex of $Y$ are within the distance $R$.
Thus, we see that Theorem \ref{pseudo} is just a special case of Theorem \ref{main}.
\end{proof}

We consider the character variety $\chi(F_g)$ of the representations of the free group $F_g$ of rank $g$ to $\pslc$.
Let $\mathcal S_g$ be the subspace of $\chi(F_g)$ consisting of Schottky representations.

\begin{thm}
\label{density}
Every point in the frontier of $\mathcal S_g$ in $\chi(F_g)$ is a limit of a sequence of primitive stable unfaithful discrete representations $\{\rho_n\}$ such that $\bh^3/\rho_n(F_g)$ is a closed hyperbolic 3-manifold for every $n$.
\end{thm}

\begin{proof}
In the following, we regard faithful discrete representations of $F_g$ as hyperbolic structures on the interior of a handlebody $H_g$ of genus $g$.
We let $S$ be the  boundary of $H_g$, and regard ending laminations or parabolic curves as lying on $S$.

By Corollary 15.1 of Canary-Culler-Hersonsky-Shalen \cite{CCHS},  in  $\Fr\mathcal S_g$, the maximal cusps, \ie geometrically finite representations without non-trivial quasi-conformal deformations, are dense.
Therefore we have only to show that any maximal cusp is a limit of a sequence of primitive stable unfaithful discrete representations corresponding to closed manifolds.

Let $\rho: F_g \rightarrow \pslc$ be a maximal cusp, and $M$ a hyperbolic 3-manifold $\bh^3/\rho(F_g)$.
We note that the union of parabolic curves of any maximal cusp is doubly incompressible, \ie has a positive lower bound for the intersection numbers with the meridians of $H_g$.
Let $c_1, \dots , c_p$ be the parabolic curves of $\rho$ regarded as lying on $S$.
Now, let $\lambda$ be the stable lamination contained in the Masur domain of $S$  of a pseudo-Anosov homeomorphism $\phi$ on $S$.
Let $\psi: F_g \rightarrow \pslc$ be a representation on the frontier of $\mathcal S_g$ with ending lamination $|\lambda|$ (the support of $\lambda$).
Now, we consider the composition of  $n$-time iterated Dehn twists on $c_1, \dots, c_p$, and denote it by $\tau_n$.
Consider the measured lamination $\tau_n(\lambda)$ and its projective class $[\tau_n(\lambda)]$.
Since $c_1 \cup \dots \cup c_p$ is doubly incompressible and $[\tau_n(\lambda)]$ converges to $[c_1\cup \dots \cup c_p]$ as $n \rightarrow \infty$, we see that $\tau_n(\lambda)$ is also doubly incompressible.
On the other hand, since $\tau_n(\lambda)$ is arational, it must be contained in the Masur domain for large $n$ as was shown in Lemma 3.4 in Lecuire \cite{Le}.
Let $\psi_n$ be a representation on the frontier of $\mathcal S_g$ with ending lamination $|\tau_n(\lambda)|$ for large $n$.

Fix $n$ and let $P_j\subset S$ be a pants decomposition converging projectively to $\tau_n(\lambda)$.
Now, we construct a $3$-manifold $M_j$ as follows.
We glue a $2$-handle to $\partial H_g$ along an annular neighbourhood (in $S$) of each component of $P_j$.
We get a $3$-manifold whose boundary is a union of spheres.
Then we glue a $3$-ball along each boundary component and denote the resulting 3-manifold by $M_j$.
It is easy to see that $S\subset M_j$ is a Heegaard surface.
We denote by $H_g^1$ the original $H_g$ and by $H_g^2(j)$ the one lying on the opposite side of $S$.
Then as $j \rightarrow \infty$, the meridian complex for $H_g^2(j)$ converges to the point at infinity represented by $\tau_n(\lambda)$ in the Gromov bordification of the curve graph of $S$.

It is known that any geodesic ray with endpoint equal to the support of a stable lamination has subsurface boundedness  (see Minsky \cite{Mi1}).
It follows that  $M_j$ satisfies the assumptions of Theorem \ref{main} for sufficiently large $j$.
Thus $M_j$ is hyperbolic and the representation $\phi_j:F_g\rightarrow \pslc$ induced by the inclusion $H_g\subset M_j$ is primitive stable. Furthermore, it follows easily from the arguments in the proof of Theorem \ref{main} and the Ending Lamination Theorem  that $\psi_n$ is the limit of $\{\phi_j\}$ as $j \rightarrow \infty$.
Hence $\psi_n$ is a limit of a sequence of primitive stable unfaithful discrete representations.

Since $c_1\cup \dots \cup c_p$ is doubly incompressible, the representations $\psi_n$ converge in $\chi(F_g)$  as $n \rightarrow \infty$ by the main theorem of Kim-Lecuire-Ohshika \cite{KLO} and the limit is the maximal cusp $\rho$ by lower semi-continuity of the length function (see \cite{brock}).
Therefore, using a diagonal extraction, we see that $\rho$ is also a limit of primitive stable representations corresponding to closed hyperbolic 3-manifolds.
This completes the proof.
\end{proof}

\noindent     Inkang Kim\\
     School of Mathematics\\
     KIAS\\ Hoegiro 85,
     Seoul, 130-722, Korea\\
     e-mail: inkang\char`\@ kias.re.kr\\

\noindent Cyril Lecuire\\
CNRS\\
Laboratoire Emile Picard\\
Universit\'{e} Paul Sabatier\\
118 route de Narbonne\\
31062 Toulouse Cedex 4\\
e-mail: lecuire\char`\@ math.ups-tlse.fr\\

\noindent
Ken'ichi Ohshika\\
Department of Mathematics\\
Graduate School of Science\\
Osaka University\\
Toyonaka, Osaka 560-0043, Japan\\
email: ohshika\char`\@ math.sci.osaka-u.ac.jp\\

\end{document}